\documentclass[1 leqno,11pt,psfig]{amsart}
\usepackage{amssymb, amsmath,amsmath,latexsym,amssymb,amsfonts,amsbsy, amsthm,mathtools,graphicx,CJKutf8,CJKnumb,CJKulem,color}
\usepackage{graphicx,epsfig}
\usepackage{graphicx}
\usepackage{amssymb}
\usepackage{graphicx}
\usepackage{graphicx,epsfig}
\usepackage{amsmath}
\usepackage{mathrsfs}
\usepackage{amssymb}
\usepackage{amssymb,mathrsfs,amsfonts,amsmath,amsbsy,tikz}
\usepackage{fontenc}
\usepackage{textcomp}
\usepackage{amsthm,amsmath,amssymb}
\numberwithin{equation}{section}

\allowdisplaybreaks
\newcommand{\beq}{\begin{equation}}
\newcommand{\eeq}{\end{equation}}
\newcommand{\ben}{\begin{eqnarray}}
\newcommand{\een}{\end{eqnarray}}
\newcommand{\beno}{\begin{eqnarray*}}
\newcommand{\eeno}{\end{eqnarray*}}

\newtheorem{theorem}{Theorem}[section]

\newtheorem{lemma}[theorem]{Lemma}
\newtheorem{proposition}[theorem]{Proposition}

\newtheorem{remark}[theorem]{Remark}
\newtheorem{Theorem}{Theorem}[section]

\newtheorem{Corollary}[Theorem]{Corollary}

\begin{document}
\begin{CJK*}{UTF8}{gkai}
\title{Induced dynamics of non-autonomous dynamical systems}

\author{Hua Shao}
\address{Department of Mathematics, Nanjing University of Aeronautics and Astronautics,
 Nanjing 211106, P. R. China}
 \address{Key Laboratory of Mathematical Model ling and High Performance Computing of Air Vehicles,
 Nanjing University of Aeronautics and Astronautics, MIIT, Nanjing 211106, P. R. China}

\email{huashao@nuaa.edu.cn}

\date{\today}

\maketitle

\begin{abstract}
Let $f_{0,\infty}=\{f_n\}_{n=0}^{\infty}$ be a sequence of continuous self-maps on a compact metric space $X$.
The non-autonomous dynamical system $(X,f_{0,\infty})$ induces the set-valued system $(\mathcal{K}(X),
\bar{f}_{0,\infty})$ and the fuzzified system $(\mathcal{F}(X),\tilde{f}_{0,\infty})$. We prove that under some 
natural conditions, positive topological entropy of $(X,f_{0,\infty})$ implies infinite entropy of 
$(\mathcal{K}(X),\bar{f}_{0,\infty})$ and $(\mathcal{F}(X),\tilde{f}_{0,\infty})$, respectively; and zero 
entropy of $(S^1,f_{0,\infty})$ implies zero entropy of some invariant subsystems of 
$(\mathcal{K}(S^1),\bar{f}_{0,\infty})$ and $(\mathcal{F}(S^1),\tilde{f}_{0,\infty})$, respectively. 
We confirm that $(\mathcal{K}(I), \bar{f})$ and $(\mathcal{F}(I), \tilde{f})$ have infinite 
entropy for any transitive interval map $f$. In contrast, we construct a transitive non-autonomous system 
$(I, f_{0,\infty})$ such that both $(\mathcal{K}(I), \bar{f}_{0,\infty})$ and $(\mathcal{F}(I), \tilde{f}_{0,\infty})$ 
have zero entropy. We obtain that $(\mathcal{K}(X),\bar{f}_{0,\infty})$
is chain weakly mixing of all orders if and only if $(\mathcal{F}^1(X),\tilde{f}_{0,\infty})$ is so, and
chain mixing (resp. $h$-shadowing and multi-$\mathscr{F}$-sensitivity) among $(X,f_{0,\infty})$, 
$(\mathcal{K}(X),\bar{f}_{0,\infty})$ and $(\mathcal{F}^1(X),\tilde{f}_{0,\infty})$ are equivalent, 
where $(\mathcal{F}^1(X),\tilde{f}_{0,\infty})$ is the induced normal fuzzification. 
\end{abstract}



\section{Introduction}
By a non-autonomous dynamical system (simply, NDS) we mean a pair $(X,f_{0,\infty})$, where $X$ is a compact metric space
endowed with a metric $d$ and $f_{0,\infty}=\{f_n\}_{n=0}^{\infty}$ is a sequence of continuous self-maps on $X$.
For any $x_0\in X$, the positive orbit $\{x_n\}_{n=0}^{\infty}$ of $(X,f_{0,\infty})$ starting from $x_0$ is defined
by $x_n=f_0^n(x_0)$, where $f_0^n=f_{n-1}\circ\cdots\circ f_{0}$, $n\geq 1$. In the case that $f_n=f$ for all $n\geq0$,
$(X, f_{0,\infty})$ becomes the classical dynamical system $(X,f)$.

In many cases, it is not enough to only know the trajectory of a single point in $X$. One also needs to know a group of points,
that is, how a set moves. For example, when studying collective behaviors in biology, such as mass migration occurring
in birds or mammals. Though the rules in these collective animal behaviors seem simple, tornado effects, thunder and lightning,
changes in nutrient supply, or other natural factors may affect them in a certain way. Motivated by this phenomena, we consider
the non-autonomous set-valued dynamical system $(\mathcal{K}(X),\bar{f}_{0,\infty})$, where $\mathcal{K}(X)$ is the hyperspace
of all nonempty compact subsets of $X$, $\bar{f}_{0,\infty}=\{\bar{f}_n\}_{n=0}^{\infty}$ and $\bar{f}_n$ is defined by
$\bar{f}_n(A)=f_n(A)$ for any $A\in\mathcal{K}(X)$. The Hausdorff metric on $\mathcal{K}(X)$ is defined by
\[D_X(A, B)=\max\left\{\sup_{a\in A}d(a, B), \sup_{b\in B}d(b, A)\right\},\;A,B\in\mathcal{K}(X),\]
where $d(a,B)=\inf_{b\in B}{d(a, b)}$.
It is known that $(\mathcal{K}(X),D_X)$ is a compact metric space if and only if $(X, d)$ is a compact
metric space \cite{Nadler}. The natural topology induced by $D_X$ is the Vietoris topology, which is equivalently
generated by the basis consisting of all sets of the form $\langle U_1,\cdots,U_k\rangle$,
where $U_1,\cdots,U_k$ are nonempty open subsets of $X$ and
\[\langle U_1,\cdots,U_k\rangle=\{B\in\mathcal{K}(X): B\subset\cup_{i=1}^{k}U_i, \;B\cap U_{i}\neq\emptyset,\;1\leq i\leq k\}.\]

Does individual chaos imply collective chaos, and conversely? This is a fundamental question.
Bauer and Sigmund \cite{Bauer} studied systematically what chaotic properties of $(X, f)$ can be
carried over to $(\mathcal{K}(X),\bar{f})$ and vice vera. In particular, they proved that positive
topological entropy of $(X,f)$ implies infinite topological entropy of $(\mathcal{K}(X),\bar{f})$.
They also studied the other induced system $(M(X),\hat f)$ on the space $M(X)$
consisting of all Borel probability measures with the Prohorov metric and $\hat f$ is the induced map
on $M(X)$. The readers are referred to \cite{Bauer,Bernardes,Glasner,Li13,Li17,Liu,ShaoJDDE} and references therein
for the dynamics of $(M(X),\hat f)$. It is worthy to mention that Glasner and Weiss \cite{Glasner} proved that
$(X, f)$ has zero topological entropy if and only if $(M(X), \hat f)$ has zero topological entropy,
which demonstrates a big difference of dynamics between $(\mathcal{K}(X),\bar{f})$ and $(M(X),\hat f)$.

Following Bauer and Sigmund's work, there are many studies on the interactions of topological dynamics between $(X, f)$
and $(\mathcal{K}(X),\bar{f})$ (see \cite{Akin,Bernardes-Vermersch2015,FG15,FG16,Guirao,Lamparta,Roman03,Sharman10,Wu15}
and references therein). For example, Lampart and Raith proved that the topological entropy of $\bar{f}$ is zero or infinity
for a homeomorphism $f$ on the compact interval (resp. unit circle) \cite{Lamparta}.
Fern\'andez et al. obtained that $(X, f)$ is chain mixing if and only if $(X, f)$ is chain weakly mixing
if and only if $(\mathcal{K}(X),\bar{f})$ is chain transitive \cite{FG15}, they also get that $(X, f)$
has shadowing  if and only if $(\mathcal{K}(X),\bar{f})$ has shadowing  \cite{FG16}.
Wu et al. proved that $(X, f)$ is $\mathscr{F}$-sensitive if and only if $(\mathcal{K}(X),\bar{f})$
is so when $\mathscr{F}$ is a filter \cite{Wu15}. For NDSs, R. Vasisht and R. Das
obtained syndetic sensitivity (resp. multi-sensitivity and shadowing) between $(X,f_{0,\infty})$ and
$(\mathcal{K}(X),\bar{f}_{0,\infty})$ are equivalent \cite{Vasisht}. We proved that $(X,f_{0,\infty})$
is topological mixing (resp. mild mixing, topological exact, specification property and property $P$)
if and only if $(\mathcal{K}(X),\bar{f}_{0,\infty})$ is so \cite{Shao19}.

Nevertheless, when dealing with non-deterministic issues, such as demographic fuzziness, environmental
fuzziness and life expectancy, the fuzzy systems can be considered \cite{Barros00}.
Let $\mathcal{F}(X)$ be the family of all upper semi-continuous fuzzy sets $u: X\to I=[0,1]$,
$\tilde{f}_{0,\infty}=\{\tilde{f}_n\}_{n=0}^{\infty}$ and $\tilde{f}_n: \mathcal{F}(X)\to\mathcal{F}(X)$
be the Zadeh's extension of $f_n$ to $\mathcal{F}(X)$ defined by
\begin{equation*}
		\tilde{f}_n(u)(x)=\left\{\begin{array}{ll}
		\sup_{z\in f_{n}^{-1}(x)}u(z),& f_{n}^{-1}(x)\neq\emptyset,\\
        0, & f_{n}^{-1}(x)=\emptyset.\\
		\end{array}\right.
\end{equation*}
Then $(\mathcal{F}(X),\tilde{f}_{0,\infty})$ is the induced fuzzified system of $(X,f_{0,\infty})$.
For any $u\in\mathcal{F}(X)$, denote
\[[u]_{\alpha}=\{x\in X: u(x)\geq\alpha\},\;\forall\;\alpha\in(0,1],\]
and the support of $u$ is
\[[u]_{0}=\overline{\{x\in X: u(x)>0\}}.\]
Let $\mathcal{F}^{1}(X)$ denote the system of all normal fuzzy sets on $X$; that is,
\[
\mathcal{F}^{1}(X)=\{u\in\mathcal{F}(X): [u]_\alpha\in\mathcal{K}(X), \;\forall \;\alpha\in I\}.
\]
Then $(\mathcal{F}^{1}(X),\tilde{f}_{0,\infty})$ is the induced normal fuzzified system of $(X,f_{0,\infty})$.
The levelwise metric $d_{\infty}$ on $\mathcal{F}(X)$ is defined by
\begin{equation*}
d_{\infty}(u,v)=\sup_{\alpha\in(0,1]}D_X([u]_{\alpha},[v]_{\alpha}),\;u,v\in\mathcal{F}(X).
\end{equation*}
Note that $(\mathcal{F}(X),d_{\infty})$ and $(\mathcal{F}^{1}(X),d_{\infty})$ are complete
but not compact metric spaces \cite{Kupka11}.

Many authors investigated what dynamics of $(X,f)$ can be inherited by $(\mathcal{F}(X),\tilde{f})$
and conversely how the dynamical behaviors of $(\mathcal{F}(X),\tilde{f})$ affect those of $(X,f)$
(see \cite{Canovas,Kupka1,Kupka11,Roman08,Wu18} and references therein).
It is also natural to consider the interrelations among the dynamics of $(X,f_{0,\infty})$ and
its two induced systems $(\mathcal{K}(X),\bar{f}_{0,\infty})$ and $(\mathcal{F}(X),\tilde{f}_{0,\infty})$.
Wu proved that topological mixing (resp. weak mixing) among $(X,f)$, $(\mathcal{K}(X),\bar{f})$
and $(\mathcal{F}^{1}(X),\tilde{f})$ are equivalent \cite{Wu17}, and $(X,f)$ is chain mixing
if and only if $(\mathcal{K}(X),\bar{f})$ is chain transitive if and only if $(\mathcal{F}^{1}(X),\tilde{f})$ is
chain transitive \cite{Wu18}. Recently, we obtained the equivalence of topological mixing (resp., mild mixing,
cofinite sensitivity and syndetic sensitivity) among $(X,f_{0,\infty})$,
$(\mathcal{K}(X),\bar{f}_{0,\infty})$ and $(\mathcal{F}^{1}(X),\tilde{f}_{0,\infty})$ \cite{Shao21}.
In this paper, we present some new results on this topic as follows.

(1) Positive topological entropy of $(X, f_{0,\infty})$ implies infinite topological entropy of
$(\mathcal{F}(X),\tilde{f}_{0,\infty})$ and $(\mathcal{K}(X),\bar{f}_{0,\infty})$, respectively,
if $f_{0,\infty}$ is equi-continuous, see Theorem \ref{a}.

(2) Zero topological entropy of $(S^1, f_{0,\infty})$ implies zero topological entropy of some
invariant subsystems of $(\mathcal{F}(X),\tilde{f}_{0,\infty})$ and $(\mathcal{K}(X),\bar{f}_{0,\infty})$,
respectively, if $f_{0,\infty}$ is a sequence of equi-continuous and orientation preserving maps on the unit circle $S^{1}$,
see Theorem \ref{s}.

(3) Both $(\mathcal{K}(I),\bar{f})$ and $(\mathcal{F}(I),\tilde{f})$ have infinite topological entropy
for any transitive interval map $f$. In contrast, there exists a transitive non-autonomous
dynamical system $(I, f_{0,\infty})$ such that both $(\mathcal{K}(I),\bar{f}_{0,\infty})$ and $(\mathcal{F}(I),\tilde{f}_{0,\infty})$
have zero topological entropy, see Theorem \ref{transitivity-entropy}.

(4) Chain weakly mixing of order $n$: $(\mathcal{F}^{1}(X),\tilde{f}_{0,\infty})\Rightarrow
(\mathcal{K}(X),\bar{f}_{0,\infty})\Rightarrow(X,f_{0,\infty})$, see Theorem \ref{ctransitive};

chain weakly mixing of all orders: $(\mathcal{K}(X),\bar{f}_{0,\infty})\Leftrightarrow(\mathcal{F}^{1}(X),\tilde{f}_{0,\infty})$,
if $f_{0,\infty}$ is equi-continuous, see Theorem \ref{ctransitive2}.

(5) Chain mixing among $(X, f_{0,\infty})$, $(\mathcal{K}(X),\bar{f}_{0,\infty})$ and $(\mathcal{F}^{1}(X),$ $\tilde{f}_{0,\infty})$
are equivalent, if $f_{0,\infty}$ is equi-continuous, see Theorem \ref{cmixing}.

(6) $(X, f_{0,\infty})$ has shadowing  if and only if $(\mathcal{K}(X),\bar{f}_{0,\infty})$ has shadowing
if and only if $(\mathcal{F}^{1}(X),\tilde{f}_{0,\infty})$ has finite shadowing, see Theorem \ref{f1}.

(7) $h$-shadowing  among $(X,f_{0,\infty})$, $(\mathcal{K}(X),f_{0,\infty})$ and $(\mathcal{F}^{1}(X),\tilde{f}_{0,\infty})$
are equivalent, see Theorem \ref{hshadowing}.

(8) Multi-$\mathscr{F}$-sensitivity among $(X,f_{0,\infty})$, $(\mathcal{K}(X),f_{0,\infty})$ and
$(\mathcal{F}^{1}(X),\tilde{f}_{0,\infty})$ are equivalent, where $\mathscr{F}$ is a Furstenberg family, see Theorem \ref{sensitive}.

The rest of the paper is organized as follows. Sections 2--5 investigate the connections of topological
dynamics among $(X,f_{0,\infty})$, $(\mathcal{K}(X),\bar{f}_{0,\infty})$ and $(\mathcal{F}(X),\tilde{f}_{0,\infty})$.
Section 2 focuses on topological entropy, especially on the positive entropy and zero entropy.
In Section 3, various chain properties (chain transitivity, chain weak mixing and chain mixing) are studied.
Shadowing properties and $\mathscr{F}$-sensitivity are discussed in Sections 4 and 5, respectively.

\section{Topological entropy}

Let $\mathbf{N}$ and $\mathbf{Z^{+}}$ denote the set of all nonnegative integers and all positive
integers, respectively, and let $B_{d}(x,\varepsilon)$ and $\bar{B}_{d}(x,\varepsilon)$ denote the
open and closed balls of radius $\varepsilon>0$ centered at $x\in X$, respectively.

Recall the definitions of topological entropy of $(X,f_{0,\infty})$ introduced in \cite{Kolyada96}.
For open covers $\mathcal{A}_{1},\cdots,\mathcal{A}_{n},\; \mathcal{A}$ of $X$, denote
$\bigvee_{i=1}^{n}\mathcal{A}_{i}:=\{\bigcap_{i=1}^{n}A_{i}: A_{i}\in\mathcal{A}_{i},\  1\leq i\leq n\}$,
$f_{i}^{-n}(\mathcal{A}):=\{f_{i}^{-n}(A): A\in\mathcal{A}\}$ and $\mathcal{A}_{i}^{n}:=\bigvee_{j=0}^{n-1}f_{i}^{-j}(\mathcal{A})$, where
$f_{i}^{n}=f_{i+n-1}\circ\cdots\circ f_{i+1}\circ f_{i}$ and $f_{i}^{-n}=f_{i}^{-1}\circ f_{i+1}^{-1}\circ\cdots\circ f_{i+n-1}^{-1}$. Let $\mathcal{N}(\mathcal{A})$ be the minimal possible cardinality of all subcovers chosen from $\mathcal{A}$.
Then the topological entropy of $(X,f_{0,\infty})$ on the cover $\mathcal{A}$ is defined by
\[h(f_{0, \infty},\mathcal{A}):=\limsup_{n\rightarrow\infty}\log\mathcal{N}(\mathcal{A}_{0}^{n})/n,\]
and the topological entropy of $(X,f_{0,\infty})$ is defined by
\[h(f_{0, \infty}):=\sup\{h(f_{0, \infty},\mathcal{A}): \mathcal{A}\; {\rm is\; an\; open\; cover\; of\; X\;}\}.\]

Another equivalent definition is based on separated sets and spanning sets.
Let $n\in\mathbf{Z^{+}}$ and $\epsilon>0$. A subset $E\subset X$ is called $(n,\epsilon)$-separated if,
for any $x\neq y\in E$, there exists $0\leq k\leq n-1$ such that $d(f_{0}^{k}(x),f_{0}^{k}(y))>\epsilon$.
A subset $F\subset X$ is called $(n,\epsilon)$-spans another set $K\subset X$ if, for any $y\in K$,
there exists $x\in F$ such that $d(f_{0}^{k}(x),f_{0}^{k}(y))\leq\epsilon$ for any $0\leq k\leq n-1$.
Let $\Lambda\subset X$ and $s_{n}(\epsilon,\Lambda,f_{0,\infty})$ be the maximal cardinality of an
$(n,\epsilon)$-separated set in $\Lambda$ and $r_{n}(\epsilon,\Lambda,f_{0,\infty})$ be the minimal
cardinality of an $(n,\epsilon)$-spanning set in $\Lambda$. Denote
\[
h(f_{0,\infty},\Lambda):=\lim_{\epsilon\to0}\limsup_{n\to\infty}\frac{1}{n}\log s_{n}(\epsilon,\Lambda,f_{0,\infty})
=\lim_{\epsilon\to0}\limsup_{n\to\infty}\frac{1}{n}\log r_{n}(\epsilon,\Lambda,f_{0,\infty}).
\]
If $\Lambda=X$, then $h(f_{0,\infty}):=h(f_{0,\infty},X)$ is called the topological entropy
of $(X,f_{0,\infty})$. Clearly, $h(f_{0,\infty},\Lambda)\leq h(f_{0,\infty})$ for any nonempty subset
$\Lambda\subset X$.

\subsection{Some basic properties of topological entropy in NDSs}

Let $g_{0,\infty}=\{g_n\}_{n=0}^{\infty}$ be a sequences of continuous self-maps on a compact metric space $(Y,\rho)$.
$(X\times Y,f_{0,\infty}\times g_{0,\infty})$ is the product of $(X,f_{0,\infty})$ and $(Y,g_{0,\infty})$, where
$f_{0,\infty}\times g_{0,\infty}=\{f_n\times g_n\}_{n=0}^{\infty}$ and the metric $d\times\rho$ on $X\times Y$ is defined
by $d\times\rho((x_1,y_1),(x_2,y_2))=\max\{d(x_1,x_2),\rho(y_1,y_2)\}$ for any $(x_1,y_1),(x_2,y_2)\in X\times Y$.
Then we have the following result.

\begin{proposition}
$$h(f_{0,\infty}\times g_{0,\infty})\leq h(f_{0,\infty})+h(g_{0,\infty}).$$
\end{proposition}

\begin{proof}
Let $n\in\mathbf{Z^{+}}$ and $\epsilon>0$. Suppose that $F_1$ is an $(n,\epsilon)$-spanning set for $X$ with respect to
$f_{0,\infty}$ with $|F_1|=r_n(\epsilon,X,f_{0,\infty})$ and $F_2$ is an $(n,\epsilon)$-spanning set for $Y$ with respect
to $g_{0,\infty}$ with $|F_2|=r_n(\epsilon,Y,g_{0,\infty})$, where $|F_i|$ is the cardinality of the set $F_i$, $i=1,2$.
Then $F_1\times F_2$ is an $(n,\epsilon)$-spanning set for $X\times Y$ with respect to $f_{0,\infty}\times g_{0,\infty}$.
Thus,
\[r_n(\epsilon,X\times Y,f_{0,\infty}\times g_{0,\infty})\leq r_n(\epsilon,X,f_{0,\infty})r_n(\epsilon,Y,g_{0,\infty}).\]
Hence,
\begin{equation*}
\begin{split}
h(f_{0,\infty}\times g_{0,\infty})
&=\lim_{\epsilon\to0}\limsup_{n\to\infty}\frac{1}{n}\log r_{n}(\epsilon,X\times Y,f_{0,\infty}\times g_{0,\infty})\\
&\leq\lim_{\epsilon\to0}\limsup_{n\to\infty}\frac{1}{n}\log r_{n}(\epsilon,X,f_{0,\infty})
+\lim_{\epsilon\to0}\limsup_{n\to\infty}\frac{1}{n}\log r_{n}(\epsilon,Y,g_{0,\infty})\\
&=h(f_{0,\infty})+h(g_{0,\infty}).
\end{split}
\end{equation*}
\end{proof}

If $X=Y$ and $f_{0,\infty}=g_{0,\infty}$, then one gets the following better result.

\begin{lemma}[Proposition 3.1, \cite{Liu}]\label{l}
$h(f_{0,\infty}^{(k)})=kh(f_{0,\infty})$ for any $k\in\mathbf{Z^{+}}$,
where $X^k=\underbrace{X\times\cdots\times X}_{k\;\;times}$, $f_{0,\infty}^{(k)}=\{f_{n}^{(k)}\}_{n=0}^{\infty}$ and $f_{n}^{(k)}=\underbrace{f_n\times\cdots\times f_n}_{k \;\;times}$ is a continuous self-map on $X^k$.
\end{lemma}

$(X,f_{0,\infty})$ is said to be topologically $\pi$-equi-semiconjugate to $(Y,g_{0,\infty})$ if there exists a
continuous and surjective map $\pi$ from $X$ to $Y$ such that $\pi\circ f_n=g_n\circ \pi$ for all $n\geq0$, and $\pi$
is called a topological equi-semiconjuacy between $(X,f_{0,\infty})$ and $(Y,g_{0,\infty})$. If there exists $c>0$
such that $\sup_{y\in Y}|\pi^{-1}(y)|\leq c$, then $\pi$ is called finite-to-one.

By Theorems B and C in \cite{Kolyada96}, we have the following result.

\begin{proposition}\label{e}
Let $f_{0,\infty}$ and $g_{0,\infty}$ be equi-continuous on $X$ and $Y$, respectively.
If there exists a finite-to-one map $\pi: X\to Y$ such that $(X,f_{0,\infty})$ is
topologically $\pi$-equi-semiconjugate to $(Y,g_{0,\infty})$, then $h(f_{0,\infty})=h(g_{0,\infty})$.
\end{proposition}

Let $\alpha$ be an open cover of $X$. A Lebesgue number of $\alpha$ is some $\delta>0$ satisfying that
for any $E\subset X$, $d(E)<\delta$ implies that $E\subset A$ for some $A\in\alpha$.
The diameter of $\alpha$ is $d(\alpha)=\sup_{A\in\alpha}d(A)$, where $d(A)=\sup_{x,y\in A} d(x,y)$.

\begin{proposition}\label{0}
Let $\{\alpha_n\}_{n=0}^{\infty}$ be a sequence of open covers of $X$ with $d(\alpha_n)\to0$ as $n\to\infty$.
Then $\lim_{n\to\infty}h(f_{0,\infty},\alpha_n)=h(f_{0,\infty})$.
\end{proposition}

\begin{proof}
Suppose that $h(f_{0,\infty})<\infty$. For any $\epsilon>0$, there exists an open cover $\beta$ such that
$h(f_{0,\infty},\beta)>h(f_{0,\infty})-\epsilon$. Let $\delta>0$ be a Lebesgue number of $\beta$.
Since $d(\alpha_n)\to0$ as $n\to\infty$, there exists $N\in\mathbf{Z^{+}}$ such that $d(\alpha_n)<\delta$
for any $n\geq N$. Then $\alpha_n$ is a refinement of $\beta$ for any $n\geq N$. Thus,
\[h(f_{0,\infty})-\epsilon<h(f_{0,\infty},\beta)\leq h(f_{0,\infty},\alpha_n)\leq h(f_{0,\infty}),\;n\geq N.\]
Hence, $\lim_{n\to\infty}h(f_{0,\infty},\alpha_n)=h(f_{0,\infty})$.

Suppose that $h(f_{0,\infty})=\infty$. For any $M>0$, there exists an open cover $\gamma$ such that
$h(f_{0,\infty},\gamma)>M$. Let $\delta>0$ be a Lebesgue number of $\gamma$. By the fact that $d(\alpha_n)\to 0$ as $n\to\infty$,
there exists $N\in\mathbf{Z^{+}}$ such that
\[h(f_{0,\infty},\alpha_n)\geq h(f_{0,\infty},\gamma)>M,\;n\geq N,\]
which implies that $h(f_{0,\infty},\alpha_n)\to\infty$ as $n\to\infty$.
\end{proof}

\begin{Corollary}
$h(f_{0,\infty})=\lim_{\delta\to0}\{\sup h(f_{0,\infty},\alpha): d(\alpha)<\delta\}$.
\end{Corollary}

The following proposition is an interesting comparison between non-autonomous and classical (i.e. autonomous) dynamical systems.

\begin{proposition}
{\rm(i)} {\rm(Proposition 2, \cite{Lamparta})} Let $f$ be a continuous map on a compact metric space $(X,d)$. Assume that for any $x\in X$,
there exists a fixed point $y$ of $f$ such that $\lim_{n\to\infty}f^{n}(x)=y$. Then $h(f)=0$.

{\rm(ii)} {\rm(Theorem 4, \cite{BO12})} There exists a non-autonomous discrete systems $([0,1], f_{0,\infty})$ such that $0, 1$ are fixed points and all others are asymptotic to $0$, but $h(f_{0,\infty})\geq\log 2$.
\end{proposition}

\begin{remark}
The example given in {\rm(}\cite{BO12},Theorem 4{\rm)} is to show positive topological entropy does not necessarily imply Li-Yorke chaos for non-autonomous dynamical systems, which is of big difference from that of classical case.
\end{remark}

\subsection{Some relations of topological entropy between NDS and its induced systems}
For any $u\in\mathcal{F}(X)$, define the endograph of $u$ by
\begin{equation*}
end(u):=\{(x,a)\in X\times I: u(x)\geq a\}.
\end{equation*}
Let $\emptyset_{X}$ be the empty fuzzy set defined by $\emptyset_{X}(x)=0$ for any $x\in X$,
and $\mathcal{F}^0(X)$ be the set of all nonempty upper semicontinuous fuzzy sets. Define
the endograph metric $d_E$ by
\begin{equation*}
d_{E}(u,v)=D_{X\times I}(end(u),end(v)),\;u,v\in\mathcal{F}^{0}(X),
\end{equation*}
where $D_{X\times I}$ is the Hausdorff metric on $\mathcal{K}(X\times I)$. Moreover, we use
\[d_{E}(\emptyset_{X},\emptyset_{X})=0\]
and
\[d_{E}(\emptyset_{X},u)=D_{supp(u)\times I}(end(\emptyset_{X}),end(u)),\;u\in\mathcal{F}^{0}(X),\]
in order to obtain a metric on the whole space $\mathcal{F}(X)$ and $(\mathcal{F}(X), d_{E})$ is a compact metric space \cite{Kupka11}.
Since we discuss the topological entropy of $(\mathcal{F}(X),\tilde{f}_{0,\infty})$ here, $\mathcal{F}(X)$
is equipped with the metric $d_{E}$ in this section, and $\mathcal{F}(X)$
is equipped with the metric $d_{\infty}$ in Sections 3--5.

\begin{theorem}\label{a}
Let $f_{0,\infty}$ be equi-continuous on $X$. Then $h(f_{0,\infty})>0$ implies that
$h(\tilde{f}_{0,\infty})=h(\bar{f}_{0,\infty})=+\infty$.
\end{theorem}

\begin{proof}
Fix $m\in\mathbf{Z^{+}}$.
Denote
\[\mathcal{K}_m(X)=\{C\in\mathcal{K}(X): |C|\leq m\}.\]
Define a map $\Phi: X^{m}\to \mathcal{K}_m(X)$ by
\[\Phi((x_1, x_2,\cdots,x_m))=\{x_1, x_2,\cdots,x_m\},\;(x_1, x_2,\cdots,x_m)\in X^{m}.\]
Then $\Phi$ is continuous, surjective and at most $m!$-to-one.
For any $(x_1, x_2,\cdots,x_m)\in X^{m}$,
\[\Phi\circ f_{k}^{(m)}((x_1, x_2,\cdots,x_m))=\{f_{k}(x_1),\cdots,f_{k}(x_m)\}
=\bar{f}_{k}\circ\Phi((x_1, x_2,\cdots,x_m)),\;k\geq0.\]
So, $(X^m, f_{0,\infty}^{(m)})$ is topologically equi-semiconjugate to $(\mathcal{K}_m(X),\bar{f}_{0,\infty})$.
It is easy to verify that $\{f_n^{(m)}\}_{n=0}^{\infty}$ is equi-continuous in $X^{m}$
and $\{\bar{f}_n\}_{n=0}^{\infty}$ is equi-continuous in $\mathcal{K}_m(X)$ since $\{f_n\}_{n=0}^{\infty}$ is equi-continuous in $X$.
Thus, by Lemma \ref{l} and Proposition \ref{e}, we have
\[mh(f_{0,\infty})=h(f_{0,\infty}^{(m)})=h(\bar{f}_{0,\infty},\mathcal{K}_m(X))\leq h(\bar{f}_{0,\infty}).\]
Hence, $h(f_{0,\infty})>0$ implies that $h(\bar{f}_{0,\infty})=+\infty$.
In addition, $h(\tilde{f}_{0,\infty})\geq h(\bar{f}_{0,\infty})=+\infty$ by Theorem 5.2 in \cite{Shao21}.
\end{proof}

Define $\chi_{K}\in\mathcal{F}^{1}(X)$ by $\chi_{K}(x)=1$ if $x\in K$ and $\chi_{K}(x)=0$
if $x\notin K$. We briefly write $\chi_{z}$ if $K$ is the single set $\{z\}$. Let $\mathcal{C}(S^1)$
be the hyperspace of all continua, that is nonempty compact connected subsets of the unit circle $S^{1}$.
Note that $\mathcal{C}(S^1)$ is compact in $\mathcal{K}(S^1)$ \cite{Wicks91}. Denote $\widetilde{\mathcal{C}}(S^1):=
\{\chi_{K}: K\in\mathcal{C}(S^1)\}$. Then $\widetilde{\mathcal{C}}(S^1)$ is a compact subset of $\mathcal{F}^{1}(X)$
and invariant under $\tilde{f}_{0,\infty}$; that is, $\tilde{f}_n(\widetilde{\mathcal{C}}(S^1))\subset\widetilde{\mathcal{C}}(S^1)$
for any $n\geq0$ \cite{Shao21}. Motivated by Theorem 4 in \cite{Lamparta}, we get the following result.

\begin{theorem}\label{s}
Let $f_{0,\infty}$ be a sequence of equi-continuous and orientation preserving maps of $S^{1}$.
Then $h(\tilde{f}_{0,\infty},\widetilde{\mathcal{C}}(S^1))=h(\bar{f}_{0,\infty},\mathcal{C}(S^1))=h(f_{0,\infty},S^{1})=0$.
\end{theorem}

\begin{proof}
$h(f_{0,\infty},S^{1})=0$ by Theorem D in \cite{Kolyada96}.
It is to show that $h(\bar{f}_{0,\infty},\mathcal{C}(S^1))=0$. Denote
\[\Theta:=\{(a,S^1): a\in S^1\}\cup\{(a,[a,b]): a,b\in S^1\},\]
where $[a,b]$ is the collection of all the points lying in counterclockwise sense between $a$ and $b$ and $[a,b]=\{a\}$ if $a=b$.
Then $\Theta$ is a closed subset of $S^1\times\mathcal{C}(S^1)$.
For any $n\geq0$, define a map $F_n: \Theta\to\Theta$ by
\[F_n(a,S^1)=(f_{n}(a),S^1),\]
and
\[F_n(a,[a,b])=(f_{n}(a),f_{n}([a,b]))=(f_{n}(a),[f_{n}(a),f_{n}(b)]),\;a,b\in S^1.\]
Since $f_{0,\infty}$ is equi-continuous, $F_{0,\infty}:=\{F_n\}_{n=0}^{\infty}$
is also equi-continuous in $\Theta$.
Let $\pi: \Theta\to\mathcal{C}(S^1)$ be the projection map.
Then $\pi$ is a topological equi-semiconjugacy between $(\Theta,F_{0,\infty})$ and $(\mathcal{C}(S^1),\bar{f}_{0,\infty})$.
Theorem B in \cite{Kolyada96} ensures that
\begin{align}\label{t}
h(F_{0,\infty})\geq h(\bar{f}_{0,\infty},\mathcal{C}(S^1)).
\end{align}
Define a map $\Psi:\Theta\to S^1\times S^1$ by
\[\Psi((a,S^1))=(a,a),\;a\in S^1,\]
and
\[\Psi((a,[a,b]))=(a,b),\;a,b\in S^1.\]
Then $\Psi$ is continuous and surjective (at most two-to-one) in $\Theta$ and satisfies that
\[\Psi\circ F_n=f_n\times f_n\circ\Psi,\;n\geq0.\]
Thus, $(\Theta,F_{0,\infty})$ is topologically equi-semiconjugate to $(S^1\times S^1,f_{0,\infty}^{(2)})$
where $f_{0,\infty}^{(2)}:=\{f_n\times f_n\}_{n=0}^{\infty}$.
It follows from Proposition \ref{e} that
\[h(F_{0,\infty})=h(f_{0,\infty}^{(2)})=2h(f_{0,\infty})=0.\]
Hence, $h(\bar{f}_{0,\infty},\mathcal{C}(S^1))=0$ by (\ref{t}).

Define $g:\mathcal{C}(S^1)\to\widetilde{\mathcal{C}}(S^1)$ by $g(K)=\chi_K$ for any $K\in\mathcal{C}(S^1)$.
Clearly, $g$ is bijective. With a similar method used in the proof of Theorem 5.2 in \cite{Shao21}, one can show that
$(\mathcal{C}(S^1),\bar{f}_{0,\infty})$ is topologically $g$-equi-conjugate to $(\widetilde{\mathcal{C}}(S^1),\tilde f_{0,\infty})$
Thus,
\[h(\tilde{f}_{0,\infty},\widetilde{\mathcal{C}}(S^1))=h(\bar{f}_{0,\infty},\mathcal{C}(S^1))=0.\]
\end{proof}

To proceed, we need the following lemma.

\begin{lemma}\label{uniformconverge}
$\{f_n\}_{n=0}^{\infty}$ converges uniformly to $f$ if and only if $\{\tilde{f}_n\}_{n=0}^{\infty}$
converges uniformly to $\tilde{f}$ if and only if $\{\bar{f}_n\}_{n=0}^{\infty}$
converges uniformly to $\bar{f}$.
\end{lemma}

\begin{proof}
If $\{f_n\}_{n=0}^{\infty}$ converges uniformly to $f$, then $\{\tilde{f}_n\}_{n=0}^{\infty}$
converges uniformly to $\tilde{f}$ by Proposition 2 in \cite{Kupka11}.
Conversely, if $\{\tilde{f}_n\}_{n=0}^{\infty}$ converges uniformly to $\tilde{f}$,
then for any $0<\epsilon<1$, there exists $N\in\mathbf{Z^{+}}$ such that for any $u\in\mathcal{F}(X)$,
\begin{equation}\label{def}
d_E(\tilde{f}_n(u),\tilde{f}(u))<\epsilon,\;n>N.
\end{equation}
Let $z\in X$. By (\ref{def}) and Lemma 5.1 (ii) in \cite{Shao21},
\[\min\{d(f_n(z),f(z)),1\}=d_E(\tilde{f}_n(\chi_{z}),\tilde{f}(\chi_{z}))<\epsilon,\;n>N,\]
which implies that
\[d(f_n(z),f(z))<\epsilon,\;n>N.\]
This proves that $\{f_n\}_{n=0}^{\infty}$ converges uniformly to $f$.

Suppose that $\{f_n\}_{n=0}^{\infty}$ converges uniformly to $f$. Let $\epsilon>0$.
Then there exists $N\in\mathbf{Z^{+}}$ such that $d(f_n(x),f(x))<\epsilon$ for any
$x\in X$ and $n>N$. Given $K\in\mathcal{K}(X)$. It is easy to verify that
\[D_X(f_n(K),f(K))<\epsilon,\;n>N,\]
which yields that $\{\bar{f}_n\}_{n=0}^{\infty}$ converges uniformly to $\bar{f}$.
Conversely, suppose that $\{\bar{f}_n\}_{n=0}^{\infty}$ converges uniformly to $\bar{f}$.
Then for any $\epsilon>0$, there exists $N\in\mathbf{Z^{+}}$ such that
$D_X(f_n(K),f(K))<\epsilon$ for any $K\in\mathcal{K}(X)$ and $n>N$.
Let $x\in X$. Then
\[d(f_n(x),f(x))=D_X\big(f_n(\{x\}),f(\{x\})\big)<\epsilon,\;n>N.\]
Therefore, $\{f_n\}_{n=0}^{\infty}$ converges uniformly to $f$.
\end{proof}

Recall that $(X,f_{0,\infty})$ is topologically transitive if $N(U,V)\neq\emptyset$
for any nonempty open subsets $U,V\subset X$, where $N(U,V):=\{n\in\mathbf{Z^{+}}:f_{0}^{n}(U)\cap V\neq\emptyset\}$;
it is topologically mixing if there exists $N_0\in \mathbf{Z^{+}}$ such that $N(U,V)\supset[N_0,+\infty)
\cap\mathbf{Z^{+}}$ for any nonempty open subsets $U,V\subset X$;
it is topologically exact if there exists $N_1\in \mathbf{Z^{+}}$ such that
$f_0^n(U) = X$ for any nonempty open subset $U\subset X$ and $n\geq N_1$.
Clearly, topological exactness$\Rightarrow$mixing$\Rightarrow$transitivity.

The following result shows a big difference between non-autonomous and classical (i.e. autonomous) dynamical systems.

\begin{theorem}\label{transitivity-entropy}
{\rm(i)} If $(I,f)$ is topologically transitive, then $h(\tilde{f})=h(\bar{f})=\infty$.

{\rm(ii)} There exists $(I,f_{0,\infty})$ such that $(I,f_{0,\infty})$ is topologically transitive but
$h(\tilde{f}_{0,\infty})=h(\bar{f}_{0,\infty})=0$.
\end{theorem}

\begin{proof}
(i) It follows from Corollary 3.6 in \cite{BC87} that $h(f)\geq {1\over 2}\log 2$.
Thus, $h(\tilde{f})=h(\bar{f})=\infty$ by Corollary 5.1 (ii) in \cite{Shao21}.

(ii) Let $I=[0,1]$. First, we construct a family of functions  $F_m:I\to I$ for $m>0$. Divide $I$ into $m$ intervals $J_i\triangleq[a_i,a_{i+1}]$, $0\leq i\leq m-1$, where $a_i={i\over m}$.
For any $0\leq i\leq m-1$, put $c_i,d_i\in J_i$ with $c_i=a_i+{1\over3m}$, $d_i=a_i+{2\over3m}$, $d_{-1}=0$ and $c_m=1$. The map $F_m$ is the connect-the-dots map
such that $F_m(a_i)=a_i$, $F_m(c_i)=c_{i+1}$ and $F_m(d_i)=d_{i-1}$ for   $0\leq i\leq m-1$. Then $(I,F_m)$ is topologically exact for any $m\geq 1$.
Next,  inductively define the maps $\{f_n\}_{n=0}^{\infty}$. Let $\mathcal{A}_{n}\triangleq\{[{{i}\over{2^{n}}},{{i+1}\over{2^{n}}}]:i=0,\cdots,2^{n}-1\}$.
Then there exists $s_1\geq 1$ such that $\underbrace{F_{1}\circ\cdots\circ F_{1}}_{s_1}(J)=I$ for  $J\in\mathcal{A}_{1}$. Denote $f_i\triangleq F_1$ for  $0\leq i\leq s_1-1$.
Assume that we have already defined $s_1<s_2<\cdots<s_n$ and $\{f_j\}_{j=0}^{s_n-1}$ such that $f_{0}^{s_k}(J)=I$ for  $J\in\mathcal{A}_{k}$ and $1\leq k\leq n$.
Let us define $s_{n+1}$ and  $\{f_j\}_{j=s_n}^{s_{n+1}-1}$.
%Let $\mathcal{B}_{n}\triangleq\{f_{0}^{s_n}(J): J\in\mathcal{A}_{n+1}\}$.
For any $J\in\mathcal{A}_{n+1}$, there exists
$l\geq 1$ such that $\underbrace{F_{n+1}\circ\cdots\circ F_{n+1}}_{l}(f_{0}^{s_n}(J))=I$. Denote $s_{n+1}\triangleq s_n+l$ and $f_j\triangleq F_{n+1}$ for $s_n\leq j\leq s_{n+1}-1$.
The above construction ensures that $\{f_n\}_{n=0}^{\infty}$ converges uniformly to the identity map ${id}$ on $I$,
and $f_{0}^{s_{n}}(J)=I$ for any $n\geq 1$
and $J\in\mathcal{A}_{n}$, which means that $(I, f_{0,\infty})$ is transitive.
This was first shown in Theorem 12 of \cite{BO12}. By Lemma \ref{uniformconverge}, $\{\bar{f}_n\}_{n=0}^{\infty}$ converges uniformly to $\bar{id}$ on $\mathcal{K}(I)$
and $\{\tilde{f}_n\}_{n=0}^{\infty}$ converges uniformly to $\tilde{id}$ on $\mathcal{F}(I)$.
It then follows from Theorem E in \cite{Kolyada96} that $h(\bar f_{0,\infty})\leq h(\bar {id})=0$ and
$h(\tilde{f}_{0,\infty})\leq h(\tilde{id})=0$.
\end{proof}

\begin{remark}
This result also holds for the induced system on the space $\mathcal{M}(X)$ of all Borel probability measures,
see Theorem 5.4 in \cite{ShaoJDDE}.
\end{remark}

\section{Chain properties}

Let $\delta>0$. A finite or infinite sequence of points	$\{x_0,x_1,x_2,\cdots\}$ is said to be a
$\delta$-pseudo orbit of $(X,f_{0,\infty})$ if $d(f_{i}(x_i),x_{i+1})<\delta$ for all $i\geq0$.
A finite $\delta$-pseudo orbit $\{x_0,x_1,\cdots x_n\}$ is called a $\delta$-chain with length $n$.
$(X,f_{0,\infty})$ is said to be chain transitive if for any $x,y\in X$ and any $\epsilon>0$,
there exists an $\epsilon$-chain from $x$ to $y$;
chain weakly mixing of some order $n$ if $(X^n,f_{0,\infty}^{(n)})$ is chain transitive;
chain weakly mixing of all orders if $(X^n,f_{0,\infty}^{(n)})$ is chain transitive for all $n\in\mathbf{Z^{+}}$;
chain mixing if for any $\epsilon>0$ and $x,y\in X$, there exists
$N\in\mathbf{Z^{+}}$ such that for any $n\geq N$, there exists an
$\epsilon$-chain from $x$ to $y$ with length $n$.
It follows from Lemma 3.1 in \cite{ShaoJDDE} that chain mixing$\Rightarrow$chain weakly mixing of all orders
$\Rightarrow$chain transitivity for $(X,f_{0,\infty})$.

\begin{theorem}\label{ctransitive}
If $(\mathcal{F}^1(X),\tilde{f}_{0,\infty})$ is chain weakly mixing of some order $n$, then so is $(\mathcal{K}(X),$ $\bar f_{0,\infty})$.
If $(\mathcal{K}(X),\bar f_{0,\infty})$ is chain weakly mixing of order $n$, then so is $(X,f_{0,\infty})$.
\end{theorem}

\begin{proof}
Suppose that $(\mathcal{F}^1(X),\tilde{f}_{0,\infty})$ is chain weakly mixing of order $n$.
Let $\epsilon>0$ and $(A_1,\cdots,A_n),(B_1,\cdots,B_n)\in(\mathcal{K}(X))^{n}$.
Then $(\chi_{A_1},\cdots,\chi_{A_n}),(\chi_{B_1},\cdots,\chi_{B_n})\in{(\mathcal{F}^{1}}(X))^{n}$.
Thus, there exists an $\epsilon$-chain $\big\{(\chi_{A_1},\cdots,\chi_{A_n})=(u_{1}^{0},\cdots,u_{n}^{0}),(u_{1}^{1},\cdots,u_{n}^{1}),\cdots,(u_{1}^{m-1},\cdots,$
$u_{n}^{m-1}),$ $(u_{1}^{m},\cdots,u_{n}^{m})=(\chi_{B_1},\cdots,\chi_{B_n})\big\}$
of $(({\mathcal{F}^{1}}(X))^{n},\tilde{f}_{0,\infty}^{(n)})$ with length $m$, which implies that
\[d_{\infty}\big(\tilde{f}_{j}(u_{i}^{j}),u_{i}^{j+1}\big)
=\sup_{\alpha\in(0,1]}D_X\big(\bar{f}_{j}([u_{i}^{j}]_\alpha),[u_{i}^{j+1}]_{\alpha}\big)<\epsilon,
\;1\leq i\leq n,\;0\leq j\leq m-1.\]
Thus, for any $\alpha\in(0,1]$,
\[\big\{(A_1,\cdots,A_n),([u_{1}^{1}]_\alpha,\cdots,[u_{n}^{1}]_\alpha),\cdots,([u_{1}^{m-1}]_\alpha,\cdots,[u_{n}^{m-1}]_\alpha),
(B_1,\cdots,B_n)\big\}\]
is an $\epsilon$-chain of $((\mathcal{K}(X))^{n},\bar f_{0,\infty}^{(n)})$.
Hence, $\big((\mathcal{K}(X))^{n},\bar f_{0,\infty}^{(n)}\big)$ is chain transitive and
$(\mathcal{K}(X),\bar f_{0,\infty})$ is chain weakly mixing of order $n$.

Suppose that $(\mathcal{K}(X),\bar f_{0,\infty})$ is chain weakly mixing of order $n$.
Let $\epsilon>0$ and $(x_1,\cdots,x_n),$ $(y_1,\cdots,y_n)\in X^{n}$.
Then $(\{x_1\},\cdots,\{x_n\}),(\{y_1\},\cdots,\{y_n\})\in(\mathcal{K}(X))^n$.
Thus, there exists an $\epsilon$-chain $\{(\{x_1\},\cdots,\{x_n\})=(A_{1}^{0},\cdots,A_{n}^{0}),(A_{1}^{1},\cdots,A_{n}^{1}),\cdots,
(A_{1}^{m-1},\cdots,A_{n}^{m-1}),$ $(A_{1}^{m},\cdots,$ $A_{n}^{m})=(\{y_1\},\cdots,\{y_n\})\}$ of
$\big((\mathcal{K}(X))^{n},\bar f_{0,\infty}^{(n)}\big)$.
So, there exists $a_i^{j}\in A_i^j$, $1\leq i\leq n$, $1\leq j\leq m-1$, such that
\[\big\{(x_1,\cdots,x_n),(a_1^{1},\cdots,a_{n}^{1}),\cdots,(a_1^{m-1},\cdots,a_{n}^{m-1}),(y_1,\cdots,y_n)\big\}\]
is an $\epsilon$-chain of $(X^{n},f_{0,\infty}^{(n)})$.
Hence, $(X^{n},f_{0,\infty}^{(n)})$ is chain transitive and
$(X,f_{0,\infty})$ is chain weakly mixing of order $n$.
\end{proof}

The following result is a special case of $n=1$ of Theorem \ref{ctransitive}.

\begin{Corollary}\label{ctransitive1}
If $(\mathcal{F}^1(X),\tilde{f}_{0,\infty})$ is chain transitive, then so is $(\mathcal{K}(X),\bar f_{0,\infty})$.
If $(\mathcal{K}(X),\bar f_{0,\infty})$ is chain transitive, then so is $(X,f_{0,\infty})$.
\end{Corollary}

$u\in \mathcal{F}(X)$ is a piecewise constant fuzzy set if there exists a finite number of subsets $D_i\subset X$ such
that $\cup_{i}\bar{D}_i=X$ and $u|_{int(D_i)}$ is constant. A piecewise constant fuzzy set $u$ can be given by
a strictly decreasing sequence of closed subsets $\{C_1, C_2,\cdots,C_k\}\subset \mathcal{K}(X)$ and a
strictly increasing sequence $\{\alpha_1, \alpha_2,\cdots,\alpha_k\}\subset(0, 1]$ if
\begin{equation*}
[u]_{\alpha}=C_{i+1},\;\alpha\in(\alpha_i,\alpha_{i+1}].
\end{equation*}
It is known that (see, for example, \cite{Wu17}) for any two piecewise constant fuzzy sets $u$ and $v$,
there exist non-increasing sequences of closed subsets $\{D_1,\cdots,D_k\},\{E_1,\cdots,E_k\}\subset \mathcal{K}(X)$
and a strictly increasing sequence $\{\alpha_1, \alpha_2, \cdots, \alpha_k\}\subset(0, 1]$ such that
\begin{equation*}
[u]_\alpha=D_{i+1},\; [v]_{\alpha}=E_{i+1}, \;\alpha\in(\alpha_i,\alpha_{i+1}].
\end{equation*}
Note that the set $\mathcal{F}_{P}(X)$ of all the piecewise constant fuzzy sets is dense in $\mathcal{F}^{1}(X)$ \cite{Kupka1}.
Denote $\mathcal{K}_n(X)=\{K\in\mathcal{K}(X): |K|\leq n\}$, $n\in\mathbf{Z^{+}}$.
Then $\mathcal{K}_n(X)$ is a closed subset in $\mathcal{K}(X)$ and $\mathcal{K}_{F}(X):=\bigcup_{n=1}^{\infty}\mathcal{K}_n(X)$,
the collection of all finite subsets of $X$, is dense in $\mathcal{K}(X)$ \cite{Bauer}.

$\Lambda\subset X$ is said to be invariant under $f_{0,\infty}$ if $f_n(\Lambda)\subset\Lambda$ for all $n\in\mathbf{N}$,
then $(\Lambda,f_{0,\infty})$ is said to be an invariant subsystem of $(X,f_{0,\infty})$. It is easy to see that
$\mathcal{F}_{P}(X)$ and $\mathcal{K}_{F}(X)$ are invariant under $(\mathcal{K}(X),\bar f_{0,\infty})$
and $(\mathcal{F}^1(X),\tilde{f}_{0,\infty})$, respectively.
We need the following lemma.

\begin{lemma}\label{c1}
Let $\Lambda$ be a dense subset of $X$ and invariant under $f_{0,\infty}$. Assume that $f_{0,\infty}$ is equi-continuous.
Then $(X,f_{0,\infty})$ is chain mixing {\rm(}resp. chain transitive{\rm)} if and only if $(\Lambda,f_{0,\infty})$ is chain mixing
{\rm(}resp. chain transitive{\rm)}.
\end{lemma}

\begin{proof}
It only shows the proof of chain mixing, since chain transitivity can be proved similarly.
Suppose that $(X,f_{0,\infty})$ is chain mixing.
Fix $\epsilon>0$ and $x,y\in\Lambda$. Then there exists $N\in\mathbf{Z^{+}}$ such that for any $n\geq N$,
there exists an $\epsilon/3$-chain $\{x,z_{1},\cdots,z_{n-1},y\}$ of $(X,f_{0,\infty})$.
By the equi-continuity of $f_{0,\infty}$, there exists $0<\delta<\epsilon/3$ such that for any $w_1,w_2\in X$,
\[d(w_1,w_2)<\delta\Rightarrow d(f_k(w_1),f_k(w_2))<\epsilon/3,\; k\in\mathbf{Z^{+}}.\]
Since $\Lambda$ is dense in $X$, there exists $z'_{i}\in B_{d}(z_{i},\delta)$, $1\leq i\leq n-1$.
It is easy to verify that $\{x,z'_{1},\cdots,z'_{n-1},y\}$ is an $\epsilon$-chain of $(\Lambda,f_{0,\infty})$.
Hence, $(\Lambda,f_{0,\infty})$ is chain mixing.

Suppose that $(\Lambda,f_{0,\infty})$ is chain mixing.
Let $\epsilon>0$ and $x,y\in X$. By the equi-continuity of $f_{0,\infty}$,
there exists $0<\delta<\epsilon/2$ such that for any
$w_1,w_2\in X$,
\[d(w_1,w_2)<\delta\Rightarrow d(f_k(w_1),f_k(w_2))<\epsilon/2,\; k\in\mathbf{Z^{+}}.\]
Since $\Lambda$ is dense in $X$, there exist $x'\in B_{d}(x,\delta)\cap\Lambda$ and $y'\in B_{d}(y,\delta)\cap\Lambda$.
Thus, there exists $N\in\mathbf{Z^{+}}$ such that for any $n\geq N$, there exists an $\epsilon/2$-chain
$\{x',z_{1},\cdots,z_{n-1},y'\}$ of $(\Lambda,f_{0,\infty})$.
It is easy to verify that $\{x,z_{1},\cdots,z_{n-1},y\}$ is an $\epsilon$-chain of $(X,f_{0,\infty})$.
Hence, $(X,f_{0,\infty})$ is chain mixing.
\end{proof}

\begin{theorem}\label{ctransitive2}
Let $f_{0,\infty}$ be equi-continuous. Then $(\mathcal{K}(X),\bar f_{0,\infty})$ is chain weakly mixing of all orders
if and only if $(\mathcal{F}^1(X),\tilde{f}_{0,\infty})$ is so.
\end{theorem}

\begin{proof}
It is only to show the necessity since the sufficiency can be obtained by Theorem \ref{ctransitive}.
Suppose that $(\mathcal{K}(X),\bar f_{0,\infty})$ is chain weakly mixing of all orders.
Let $n\in\mathbf{Z^{+}}$, $\epsilon>0$ and $u=(u_1,\cdots,u_n),v=(v_1,\cdots,v_n)\in(\mathcal{F}_{P}(X))^{n}$.
Then there exist $0<\alpha_1<\cdots< \alpha_k=1$ and non-increasing sequences $\{A_{i}^j\}_{i=1}^{k},\{B_{i}^j\}_{i=1}^{k}\subset\mathcal{K}(X)$, $1\leq j\leq n$, such that
\[[u_j]_{\alpha}=A_{i+1}^j,\;[v_j]_{\alpha}=B_{i+1}^j,\;\alpha\in(\alpha_i,\alpha_{i+1}].\]
By the chain transitivity of $\big((\mathcal{K}(X))^{nk},\bar f_{0,\infty}^{(nk)}\big)$,
there exist an integer $m\geq2$ and an $\epsilon$-chain from $(A_{1}^1,\cdots,A_{k}^1,\cdots,A_{1}^n,\cdots,A_{k}^n)$
to $(B_{1}^1,\cdots,B_{k}^1,\cdots,B_{1}^n,\cdots,B_{k}^n)$ with length $m$.
Thus, for any $1\leq i\leq k$ and $1\leq j\leq n$, there exists an $\epsilon$-chain
$\{A_{i}^{j},E_{1}^{i,j},\cdots,E_{m-1}^{i,j},B_{i}^{j}\}$ of $(\mathcal{K}(X),\bar f_{0,\infty})$.
For any $1\leq j\leq n$ and $1\leq l\leq m-1$, define a piecewise constant fuzzy set
$u_{l}^{j}: X\to I$ by
\[[u_{l}^{j}]_{\alpha}=\cup_{r=i+1}^{k}E_{l}^{r,j},\;\alpha\in(\alpha_i,\alpha_{i+1}].\]
By the fact that $\{A_{i}^{j},E_{1}^{i,j},\cdots,E_{m-1}^{i,j},B_{i}^{j}\}$ is an $\epsilon$-chain of
$(\mathcal{K}(X),\bar f_{0,\infty})$ and $\{A_{i}^j\}_{i=1}^{k}$
is non-increasing, one can verify that $\{u_j,u_1^{j},\cdots,u_{m-1}^{j},v_j\}$ is an $\epsilon$-chain of $(\mathcal{F}_{P}(X),\tilde{f}_{0,\infty})$ for any $1\leq j\leq n$. Thus,
$\{u,(u_1^1,\cdots,u_1^n),\cdots,(u_{m-1}^1,\cdots,u_{m-1}^n),v\}$
is an $\epsilon$-chain of $\big((\mathcal{F}_{P}(X))^{n},\tilde{f}_{0,\infty}^{(n)}\big)$.
Hence, $\big((\mathcal{F}_{P}(X))^{n},\tilde{f}_{0,\infty}^{(n)}\big)$ is chain transitive, and so is $\big(({\mathcal{F}^{1}}(X))^{n},\tilde{f}_{0,\infty}^{(n)}\big)$ by Lemma \ref{c1}.
Therefore, $(\mathcal{F}^1(X),\tilde{f}_{0,\infty})$ is chain weakly mixing of order $n$.
Since $n$ is arbitrary, $(\mathcal{F}^1(X),\tilde{f}_{0,\infty})$ is chain weakly mixing of all orders.
\end{proof}

Chain mixing of $(M(X),\hat{f}_{0,\infty})$ does not imply that of $(X,f_{0,\infty})$, even for a single map,
see Example 3.7 in \cite{ShaoJDDE}. However, it is not the case for $(\mathcal{K}(X),\bar f_{0,\infty})$ and $(\mathcal{F}^1(X),\tilde{f}_{0,\infty})$.

\begin{theorem}\label{cmixing}
Let $f_{0,\infty}$ be equi-continuous.
Then $(X,f_{0,\infty})$ is chain mixing if and only if $(\mathcal{K}(X),\bar f_{0,\infty})$ is chain mixing
if and only if $(\mathcal{F}^1(X),\tilde{f}_{0,\infty})$ is chain mixing.
\end{theorem}

\begin{proof}
With a similar proof to that of Theorem \ref{ctransitive}, one can prove that chain mixing of
$(\mathcal{F}^1(X),\tilde{f}_{0,\infty})$ implies that of $(\mathcal{K}(X),\bar f_{0,\infty})$,
and chain mixing of $(\mathcal{K}(X),\bar f_{0,\infty})$ implies that of $(X,f_{0,\infty})$.

Suppose that $(X,f_{0,\infty})$ is chain mixing.
Let $\epsilon>0$ and $A=\{x_i\}_{i=1}^{s},B=\{y_j\}_{j=1}^{t}\in\mathcal{K}_{F}(X)$.
Then there exists $N\in\mathbf{Z^{+}}$ such that for any $n\geq N$, $1\leq i\leq s$ and $1\leq j\leq t$,
there exists an $\epsilon$-chain $\{x_i,z_{1}^{i,j},\cdots,z_{n-1}^{i,j},y_j\}$ of $(X,f_{0,\infty})$.
Denote
\[E_k:=\{z_{k}^{i,j}: 1\leq i\leq s,\;1\leq j\leq t\},\;\;1\leq k\leq n-1.\]
Then $\{A,E_1,\cdots,E_{n-1},B\}$ is an $\epsilon$-chain of $(\mathcal{K}_{F}(X),\bar f_{0,\infty})$.
So, $(\mathcal{K}_{F}(X),\bar f_{0,\infty})$ is chain mixing.
Thus, $(\mathcal{K}(X),\bar f_{0,\infty})$ is chain mixing by Lemma \ref{c1}.

Suppose that $(\mathcal{K}(X),\bar{f}_{0,\infty})$ is chain mixing.
Let $\epsilon>0$ and $u,v\in\mathcal{F}_{P}(X)$. Then there exist $0<\alpha_1<\cdots<\alpha_k=1$ and
non-increasing sequences $\{P_{i}\}_{i=1}^{k},\{Q_{i}\}_{i=1}^{k}\subset\mathcal{K}(X)$ such that
\[[u]_{\alpha}=P_{i+1},\;[v]_{\alpha}=Q_{i+1},\;\alpha\in(\alpha_i,\alpha_{i+1}].\]
Thus, there exists $N\in\mathbf{Z^{+}}$ such that for any $n\geq N$ and $1\leq i\leq k$,
there exists an $\epsilon$-chain $\{P_{i}=E_{0}^i,E_{1}^i,\cdots,E_{n-1}^i,E_{n}^i=Q_{i}\}$
of $(\mathcal{K}(X),\bar f_{0,\infty})$. So,
\begin{align}\label{c22}
D_X(\bar{f}_{j}(E_{j}^i),E_{j+1}^i)<\epsilon,\;0\leq j\leq n-1,\;1\leq i\leq k.
\end{align}
For any $1\leq j\leq n-1$, define a piecewise constant fuzzy set $u_j: X\to I$ by
\[[u_j]_{\alpha}=\cup_{l=i+1}^{k}E_j^{l},\;\alpha\in(\alpha_i,\alpha_{i+1}].\]
By (\ref{c22}) and the fact that $\{P_{i}\}_{i=1}^{k}$ and $\{Q_{i}\}_{i=1}^{k}$ are non-increasing,
$\{u,u_1,\cdots,u_{n-1},v\}$ is an $\epsilon$-chain of $(\mathcal{F}_{P}(X),\tilde{f}_{0,\infty})$.
Thus, $(\mathcal{F}_{P}(X),\tilde{f}_{0,\infty})$ is chain mixing.
Hence, $(\mathcal{F}^{1}(X),\tilde{f}_{0,\infty})$ is chain mixing by Lemma \ref{c1}.
Therefore, chain mixing among $(X,f_{0,\infty})$, $(\mathcal{K}(X),\bar f_{0,\infty})$
and $(\mathcal{F}^{1}(X),\tilde{f}_{0,\infty})$ are equivalent.
\end{proof}

\section{Shadowing Properties}

$(X,f_{0,\infty})$ is said to have shadowing if for any $\epsilon>0$, there exists $\delta>0$
such that every $\delta$-pseudo orbit $\{x_n\}_{n=0}^{\infty}$ is $\epsilon$-shadowed by some $y\in X$;
that is, $d(f_{0}^{n}(y),x_n)<\epsilon$ for any $n\geq0$; it is said to have finite shadowing
if for any $\epsilon>0$, there exists $\delta>0$ such that every finite $\delta$-pseudo orbit
is $\epsilon$-shadowed by some point in $X$; it is said to have $h$-shadowing if for any $\epsilon>0$,
there exists $\delta>0$ such that every finite $\delta$-pseudo orbit $\{x_k\}_{k=0}^{n}$ is $\epsilon$-$h$-shadowed
by some $z\in X$; that is, $d(f_{0}^{k}(z),x_k)<\epsilon$ for any $0\leq k\leq n-1$ and $f_{0}^{n}(z)=x_n$.
Since $X$ is compact, it is easy to verify that for $(X,f_{0,\infty})$,
\[\text{h-shadowing }\Rightarrow \text{finite shadowing}\Leftrightarrow \text{shadowing}.\]

\begin{lemma}\label{shadowingdense}
Let $\Lambda$ be a dense subset of $X$ and invariant under $f_{0,\infty}$. If $(\Lambda,f_{0,\infty})$ has
$h$-shadowing {\rm(}resp. finite shadowing{\rm)}, then $(X,f_{0,\infty})$ has $h$-shadowing {\rm(}resp. finite shadowing{\rm)}.
\end{lemma}

\begin{proof}
It only shows the proof of $h$-shadowing, since finite shadowing can be proved similarly.
Let $\epsilon>0$. Since $(\Lambda,f_{0,\infty})$ has $h$-shadowing, there exists $\delta>0$ such that
every finite $\delta$-pseudo orbit of $(\Lambda,f_{0,\infty})$ is $\epsilon/4$-$h$-shadowed by some point in $\Lambda$.
Let $\{x_0,\cdots,x_r\}$ be a $\delta/3$-pseudo orbit in $X$. Then
\[d(f_{i}(x_i),x_{i+1})<\delta/3, \;\;0\leq i\leq r-1.\]
By the continuity of $f_{0,\infty}$, there exists $0<\zeta<\epsilon/4$ such that for any $x,y\in X$,
\[d(x,y)<\zeta\Rightarrow d(f_{i}(x),f_{i}(y))<\delta/3, \;\;0\leq i\leq r-1.\]
As $\Lambda$ is dense in $X$, there exists $y_{i,n}\in B_d(x_i,1/2^n)\cap\Lambda$ for any $0\leq i\leq r$ and $n\in\mathbf{Z^{+}}$.
Choose $N\in\mathbf{Z^{+}}$ such that $1/2^{n}<\min\{\delta/3,\zeta\}$ for any $n\geq N$.
Fix $n\geq N$. Then, for any $0\leq i\leq r-1$,
\begin{align*}
d(f_{i}(y_{i,n}),y_{i+1,n})\leq d(f_{i}(y_{i,n}),f_{i}(x_{i}))+d(f_{i}(x_{i}),x_{i+1})
+d(x_{i+1},y_{i+1,n})<\delta,
\end{align*}
which implies that $\{y_{i,n}\}_{i=0}^{r}$ is a $\delta$-pseudo orbit in $\Lambda$ for any $n\geq N$, and thus
there exists $z_n\in\Lambda$ such that
\[d(f_{0}^{i}(z_n),y_{i,n})<\epsilon/4,\;0\leq i\leq r-1,\;f_{0}^{r}(z_n)=y_{r,n}.\]
Let $z\in X$ be the limit point of $\{z_n\}_{n=N}^{+\infty}$. Then there exists $N_1\in\mathbf{Z^{+}}$ such that for any $n\geq N_1$, \[d(f_{0}^{i}(z),f_{0}^{i}(z_n))<\epsilon/2, \;\;0\leq i\leq r-1.\]
So
\begin{align*}
d(f_{0}^{i}(z),x_{i})\leq d(f_{0}^{i}(z),f_{0}^{i}(z_n))+d(f_{0}^{i}(z_n),y_{i,n})+d(y_{i,n},x_i)<\epsilon,\;
0\leq i\leq r-1,\;f_{0}^{r}(z)=x_{r}.
\end{align*}
Hence, $\{x_0,\cdots,x_r\}$ is $\epsilon$-$h$-shadowed by $z$.
Therefore, $(X,f_{0,\infty})$ has $h$-shadowing.
\end{proof}

\begin{theorem}\label{f1}
$(X,f_{0,\infty})$ has shadowing  if and only if $(\mathcal{K}(X),\bar f_{0,\infty})$ has shadowing  if and only if $(\mathcal{F}^{1}(X),\tilde{f}_{0,\infty})$ has finite shadowing.
\end{theorem}

\begin{proof}
By Corollary 4.1 in \cite{Vasisht} $(X,f_{0,\infty})$ has shadowing if and only if $(\mathcal{K}(X),\bar f_{0,\infty})$
has shadowing. Suppose that $(\mathcal{F}^{1}(X),\tilde{f}_{0,\infty})$ has finite shadowing. For any $\epsilon>0$,
there exists $\delta>0$ such that every finite $\delta$-pseudo orbit of $(\mathcal{F}^{1}(X),\tilde{f}_{0,\infty})$
is $\epsilon$-shadowed by some point in $\mathcal{F}^{1}(X)$. Let $\{z_j\}_{j=0}^{m}$ be a finite $\delta$-pseudo orbit of
$(X,f_{0,\infty})$. Then
\[d_{\infty}(\tilde{f}_{j}(\chi_{z_j}),\chi_{z_{j+1}})=d(f_{j}(z_j),z_{j+1})<\delta,\;0\leq j\leq m-1,\]
which implies that $\{\chi_{z_0},\cdots,\chi_{z_m}\}$ is a finite $\delta$-pseudo orbit of $(\mathcal{F}^{1}(X),\tilde{f}_{0,\infty})$.
Thus, there exists $u\in\mathcal{F}^{1}(X)$ such that
\[d_{\infty}(\tilde{f}_{0}^{j}(u),\chi_{z_j})<\epsilon,\;0\leq j\leq m.\]
Fix $y\in[u]_1$. Then
\[d(f_{0}^{j}(y),z_j)<\epsilon,\;0\leq j\leq m.\]
Therefore, $(X,f_{0,\infty})$ has finite shadowing, and thus has shadowing.

Suppose that $(\mathcal{K}(X),\bar{f}_{0,\infty})$ has shadowing. For any $\epsilon>0$,
there exists $\delta>0$ such that every $\delta$-pseudo orbit of $(\mathcal{K}(X),\bar{f}_{0,\infty})$
is $\epsilon$-shadowed by some point in $\mathcal{K}(X)$. Let $\{u_j\}_{j=0}^{m}\subset\mathcal{F}_{P}(X)$ be
a finite $\delta$-pseudo orbit under $\tilde{f}_{0,\infty}$. Then
\begin{align}\label{chain}
d_{\infty}(\tilde{f}_{j}(u_j),u_{j+1})<\delta,\;0\leq j\leq m-1.
\end{align}
Since $\{u_j\}_{j=0}^{m}$ are piecewise constants, there exist $0<\alpha_1<\cdots<\alpha_k=1$ and
non-increasing sequences $\{P^{j}_{i}\}_{i=1}^{k}\subset\mathcal{K}(X)$ such that
\[[u_j]_{\alpha}=P^{j}_{i+1},\;\alpha\in(\alpha_i,\alpha_{i+1}],\;0\leq j\leq m.\]
It follows from (\ref{chain}) that $\{P^{0}_{i},P^{1}_{i},\cdots,P^{m}_{i}\}$ is a $\delta$-pseudo orbit
of $(\mathcal{K}(X),\bar{f}_{0,\infty})$ for any $1\leq i\leq k$.
Thus, there exists $G_i\in\mathcal{K}(X)$ such that
\begin{align}\label{c2}
D_X(\bar{f}_{0}^{j}(G_i),P_{i}^{j})<\epsilon,\;0\leq j\leq m.
\end{align}
Define a piecewise constant fuzzy set $u: X\to I$ by
\[[u]_{\alpha}=\cup_{n=i+1}^{k}G_n,\;\alpha\in(\alpha_i,\alpha_{i+1}].\]
By (\ref{c2}) and the fact that
$\{P^{j}_{i}\}_{i=1}^{k}$ is non-increasing, we have
\[d_{\infty}(\tilde{f}_{0}^{j}(u),u_{j})<\epsilon,\;0\leq j\leq m.\]
Hence, $(\mathcal{F}_P(X),\tilde{f}_{0,\infty})$ has finite shadowing.
Therefore, $(\mathcal{F}^1(X),\tilde{f}_{0,\infty})$ has finite shadowing by Lemma \ref{shadowingdense}.
\end{proof}

\begin{remark}
Example 3.9 in \cite{ShaoJDDE} shows that shadowing of $(X, f_{0,\infty})$ is not inherited
by $(M(X), \hat f_{0,\infty})$ in general. This is another difference of dynamics between
$(\mathcal{K}(X),\bar f_{0,\infty})$ and that of $(M(X), \hat f_{0,\infty})$.
\end{remark}

\begin{theorem}\label{hshadowing}
$(X,f_{0,\infty})$ has $h$-shadowing  if and only if $(\mathcal{K}(X),\bar f_{0,\infty})$ has $h$-shadowing
if and only if $(\mathcal{F}^1(X),\tilde{f}_{0,\infty})$ has $h$-shadowing.
\end{theorem}

\begin{proof}
Suppose that $(\mathcal{K}(X),\bar f_{0,\infty})$ has $h$-shadowing.
Fix $\epsilon>0$. Then there exists $\delta>0$ such that every finite $\delta$-pseudo orbit of
$(\mathcal{K}(X),\bar f_{0,\infty})$ is $\epsilon$-$h$-shadowed by some point in $\mathcal{K}(X)$.
Let $\{x_j\}_{j=0}^{m}$ be a finite $\delta$-pseudo orbit of $(X,f_{0,\infty})$. Then
\[D_X(\bar{f}_{j}(\{x_j\}),\{x_{j+1}\})=d(f_{j}(x_j),x_{j+1})<\delta,\;0\leq j\leq m-1,\]
which implies that $\{\{x_0\},\cdots,\{x_m\}\}$ is a finite $\delta$-pseudo orbit of $(\mathcal{K}(X),\bar f_{0,\infty})$.
Thus there exists $K\in\mathcal{K}(X)$ such that
\[D_X(\bar{f}_{0}^{j}(K),\{x_j\})<\epsilon,\;0\leq j\leq m-1,\;\bar{f}_{0}^{m}(K)=\{x_m\},\]
which ensures that for any $x\in K$,
\[d(f_{0}^{j}(x),x_{j})<\epsilon,\;0\leq j\leq m-1,\;f_{0}^{m}(x)=x_m.\]
Therefore, $(X,f_{0,\infty})$ has $h$-shadowing.

Suppose that $(X,f_{0,\infty})$ has $h$-shadowing. It is to show that
$(\mathcal{K}_{F}(X),\bar{f}_{0,\infty})$ has $h$-shadowing.
Fix $\epsilon>0$. Then, there exists $\delta>0$ such that every finite $\delta$-pseudo orbit of $(X,f_{0,\infty})$ is
$\epsilon$-$h$-shadowed by some point in $X$. Let $\{K_0,\cdots,K_n\}$ be a finite $\delta$-pseudo orbit
of $(\mathcal{K}_{F}(X),\bar{f}_{0,\infty})$. Then
\begin{align}\label{hd}
D_X(\bar{f}_{i}(K_i),K_{i+1})<\delta,\;0\leq i\leq n-1.
\end{align}
Denote
\[K_{j}:=\{a_{0}^{j},a_{1}^{j},\cdots,a_{k_j}^{j}\},\;0\leq j\leq n.\]
For any $0\leq i\leq k_0$, there exists $0\leq m_{i,1}\leq k_1$ such that
$d(f_{0}(a_{i}^{0}),a_{m_{i,1}}^{1})<\delta$ by (\ref{hd}). Using (\ref{hd}) again, one has that there exists
$0\leq m_{i,2}\leq k_2$ such that $d(f_{1}(a_{m_{i,1}}^{1}),a_{m_{i,2}}^{2})<\delta$. Inductively,
there exists $0\leq m_{i,n}\leq k_n$ such that $d(f_{n-1}(a_{m_{i,n-1}}^{n-1}),a_{m_{i,n}}^{n})<\delta$
by (\ref{hd}). Thus, $\{a_{i}^{0},a_{m_{i,1}}^{1},a_{m_{i,2}}^{2},\cdots,a_{m_{i,n}}^{n}\}$
is a finite $\delta$-pseudo orbit of $(X,f_{0,\infty})$. Then there exists $b_{i}^{0}\in X$ such that
\begin{align}\label{6}
d(b_{i}^{0},a_{i}^{0})<\epsilon,\;d(f_{0}^{j}(b_{i}^{0}),a_{m_{i,j}}^{j})<\epsilon,\;1\leq j\leq n-1,\;
f_{0}^{n}(b_{i}^{0})=a_{m_{i,n}}^{n}.
\end{align}
For any $0\leq i\leq k_1$, there exists $0\leq l_{i,0}\leq k_0$ such that
$d(f_{0}(a_{l_{i,0}}^{0}),a_{i}^{1})<\delta$ by (\ref{hd}). Using (\ref{hd}) again, one has that there exists
$0\leq l_{i,2}\leq k_2$ such that $d(f_{1}(a_{i}^{1}),a_{l_{i,2}}^{2})<\delta$. Inductively,
there exists $0\leq l_{i,n}\leq k_n$ such that $d(f_{n-1}(a_{l_{i,n-1}}^{n-1}),a_{l_{i,n}}^{n})<\delta$
by (\ref{hd}). Then we obtain a $\delta$-pseudo orbit $\{a_{l_{i,0}}^{0},a_{i}^{1},a_{l_{i,2}}^{2},\cdots,a_{l_{i,n}}^{n}\}$
of $(X,f_{0,\infty})$. So, there exists $b_{i}^{1}\in X$ such that
\begin{align}\label{7}
d(b_{i}^{1},a_{l_{i,0}}^{0})<\epsilon,\;d(f_{0}(b_{i}^{1}),a_{i}^{1})<\epsilon,
\;d(f_{0}^{j}(b_{i}^{1}),a_{l_{i,j}}^{j})<\epsilon,\;2\leq j\leq n-1,\;f_{0}^{n}(b_{i}^{1})=a_{l_{i,n}}^{n}.
\end{align}
Repeating this process, for any $0\leq i\leq k_n$, we can obtain a finite $\delta$-pseudo orbit
$\{a_{t_{i,0}}^{0},a_{t_{i,1}}^{1},\cdots$ $,a_{t_{i,n-1}}^{n-1},a_{i}^{n}\}$ of $(X,f_{0,\infty})$
and there exists $b_{i}^{n}\in X$ such that
\begin{align}\label{8}
d(f_{0}^{j}(b_{i}^{n}),a_{t_{i,j}}^{j})<\epsilon,\;0\leq j\leq n-1,\;f_{0}^{n}(b_{i}^{n})=a_{i}^{n}.
\end{align}
Denote $A:=\{b_{0}^{0},\cdots,b_{k_0}^{0},b_{0}^{1},\cdots,b_{k_1}^{1},\cdots,b_{0}^{n},\cdots,b_{k_n}^{n}\}$.
So, $A\in\mathcal{K}_{F}(X)$. By (\ref{6})-(\ref{8}),
\[D_X(\bar{f}_{0}^{j}(A),K_{j})<\epsilon,\;0\leq j\leq n-1,\;\bar{f}_{0}^{n}(A)=K_{n}.\]
Thus, $(\mathcal{K}_{F}(X),\bar{f}_{0,\infty})$ has $h$-shadowing.
Hence, $(\mathcal{K}(X),\bar{f}_{0,\infty})$ has $h$-shadowing by Lemma \ref{shadowingdense}.

Using the same method in Theorem \ref{f1} one can prove that $h$-shadowing of
$(\mathcal{F}^1(X),\tilde{f}_{0,\infty})$ implies that of $(X,f_{0,\infty})$, and $h$-shadowing  of
$(\mathcal{K}(X),\bar f_{0,\infty})$ implies that of $(\mathcal{F}^1(X),\tilde{f}_{0,\infty})$.
Therefore, $h$-shadowing  among $(X,f_{0,\infty})$, $(\mathcal{K}(X),\bar f_{0,\infty})$
and $(\mathcal{F}^1(X),\tilde{f}_{0,\infty})$ are equivalent.
\end{proof}

The following result shows that shadowing plus chain mixing imply topological mixing for $(X,f_{0,\infty})$,
and thus its two induced systems are also topologically mixing.

\begin{theorem}
Let $(X,f_{0,\infty})$ have shadowing. If $(X,f_{0,\infty})$ is chain mixing, then
it is topologically mixing, and thus $(\mathcal{K}(X),\bar f_{0,\infty})$
and $(\mathcal{F}^1(X),\tilde{f}_{0,\infty})$ are topologically mixing.
\end{theorem}

\begin{proof}
Let $U$ and $V$ be two nonempty open subsets of $X$.
Fix $x\in U$ and $y\in V$. Then there exists $\epsilon>0$ such that
$B_d(x,\epsilon)\subset U$ and $B_d(y,\epsilon)\subset V$.
Since $(X,f_{0,\infty})$ has shadowing, there exists $\delta>0$ such that every $\delta$-pseudo
orbit is $\epsilon$-shadowed by some point in $X$. By chain mixing of $(X,f_{0,\infty})$, there exists $N\in\mathbf{Z^{+}}$
such that for any $k\geq N$, there exists a $\delta$-chain $\{x=x_0, x_1,\cdots, x_k=y\}$. So, there exists
$z\in X$ such that $d(f_{0}^{n}(z), x_n)<\epsilon$ for all $0\leq n\leq k$.
In particular, $z\in B_d(x,\epsilon)\subset U$ and $f_0^k (z)\in B_d(y, \epsilon)\subset V$.
Hence, $(X,f_{0,\infty})$ is topologically mixing.
Therefore, $(\mathcal{K}(X),\bar f_{0,\infty})$ and $(\mathcal{F}^1(X),\tilde{f}_{0,\infty})$
are topologically mixing by Theorem 2.11 in \cite{Shao21}.
\end{proof}

\section{$\mathscr{F}$-sensitivity}

Let $\mathscr{P}$ be the collection of all subsets of $\mathbf{N}$.
$\mathscr{F}\subset\mathscr{P}$ is a Furstenberg family if it is hereditary upwards,
i.e. $F_{1}\subset F_{2}$ and $F_1\in\mathscr{F}$ imply $F_2\in\mathscr{F}$;
and is proper if it is a proper subset of $\mathscr{P}$.
All Furstenberg families considered here are proper.
It is clear that $\mathscr{F}$ is proper if and only if $\mathbf{N}\in\mathscr{F}$ and $\emptyset\notin\mathscr{F}$.
The dual family of $\mathscr{F}$ is
\[\kappa\mathscr{F}=\{F\in\mathscr{P}: \mathbf{N}\setminus F\notin\mathscr{F}\}.\]
It is easy to verify that $\mathscr{F}$ is a Furstenberg family if and only if $\kappa\mathscr{F}$ is so and
$\kappa(\kappa\mathscr{F})=\mathscr{F}$.
For two Furstenberg families $\mathscr{F}_1$ and $\mathscr{F}_2$, set
\[\mathscr{F}_1\cdot\mathscr{F}_2=\{F_1\cap F_2 : F_1\in\mathscr{F}_1, F_2\in\mathscr{F}_2\}.\]
A Furstenberg family $\mathscr{F}$ is said to be a filter if $\mathscr{F}\cdot\mathscr{F}\subset\mathscr{F}$.

Let $\mathscr{F}$ be a Furstenberg family. $(X,f_{0,\infty})$ is $\mathscr{F}$-sensitive if there exists $\delta>0$
such that $N_{d}(x,\epsilon,\delta)\in\mathscr{F}$ for any $x\in X$ and $\epsilon>0$, where
\[
N_{d}(x,\epsilon,\delta)=\{n\in\mathbf{Z^{+}}:{\rm there\; exists}\;
y\in B_{d}(x,\epsilon)\;{\rm such\; that}
\;d\big(f_{0}^{n}(x),f_{0}^{n}(y)\big)>\delta\};
\]
it is multi-$\mathscr{F}$-sensitive if there exists $\delta>0$ such that
$\bigcap_{i=1}^{k}N_{d}(x_i,\epsilon,\delta)\in\mathscr{F}$ for any $\epsilon>0$, $k\in\mathbf{Z^{+}}$ and $x_1,\cdots,x_k\in X$.
Here, $\delta$ is called a sensitivity constant.

\begin{lemma}\label{sen1}\cite{Shao21}
Let $K\in\mathcal{K}(X)$, $u\in\mathcal{F}^{1}(X)$ and $\epsilon,\delta>0$. Then

{\rm(i)} $N_{d_{\infty}}(\chi_K,\epsilon,\delta)\subset N_{D_X}(K,\epsilon,\delta)$.

{\rm(ii)} $N_{D_X}([u]_{1},\epsilon/4,\delta)\subset N_{d_{\infty}}(u,\epsilon,\delta)$.
\end{lemma}

\begin{lemma}\label{k1}
Let $\Lambda$ be a dense subset of $X$ and invariant under $f_{0,\infty}$.
If $(\Lambda,f_{0,\infty})$ is multi-$\mathscr{F}$-sensitive, then so is $(X,f_{0,\infty})$.
\end{lemma}

\begin{proof}
Let $(\Lambda,f_{0,\infty})$ be multi-$\mathscr{F}$-sensitive with sensitivity constant $\delta>0$.
Fix $k\in\mathbf{Z^{+}}$, $x_1,\cdots,x_k\in X$ and $\epsilon>0$.
Since $\Lambda$ is dense in $X$, there exists $y_i\in B_d(x_i,\epsilon)\cap\Lambda$, $1\leq i\leq k$.
Then there exists $\epsilon_0>0$ such that $B_d(y_i,\epsilon_0)\cap\Lambda\subset B_d(x_i,\epsilon)\cap\Lambda$,
$1\leq i\leq k$. By the multi-$\mathscr{F}$-sensitivity of $(\Lambda,f_{0,\infty})$,
\[\Omega:=\bigcap_{i=1}^{k}\{n\in\mathbf{Z^{+}}:{\rm there\; exists}\;
z_i\in B_{d}(y_i,\epsilon_0)\cap\Lambda\;{\rm such\; that}
\;d\big(f_{0}^{n}(y_i),f_{0}^{n}(z_i)\big)>\delta\}\in\mathscr{F}.\]
Let $N\in\Omega$. Then for any $1\leq i\leq k$, there exists $z_i\in B_d(y_i,\epsilon_0)\cap\Lambda$ such that
\[d(f_{0}^{N}(y_i), f_{0}^{N}(z_i))>\delta,\]
which implies that
\[d(f_{0}^{N}(x_i), f_{0}^{N}(y_i))>\delta/2\; {\rm or}\; d(f_{0}^{N}(x_i), f_{0}^{N}(z_i))>\delta/2.\]
Thus, $N\in\bigcap_{i=1}^{k}N_{d}(x_i,\epsilon,\delta/2)$ and
$\Omega\subset\bigcap_{i=1}^{k}N_{d}(x_i,\epsilon,\delta/2)$.
Hence, $\bigcap_{i=1}^{k}N_{d}(x_i,\epsilon,\delta/2)\in\mathscr{F}$ and
$(X,f_{0,\infty})$ is multi-$\mathscr{F}$-sensitive.
\end{proof}

\begin{theorem}\label{sensitive}
$(X,f_{0,\infty})$ is multi-$\mathscr{F}$-sensitive if and only if $(\mathcal{K}(X),\bar f_{0,\infty})$
is multi-$\mathscr{F}$-sensitive if and only if $(\mathcal{F}^1(X),\tilde{f}_{0,\infty})$ is multi-$\mathscr{F}$-sensitive.
\end{theorem}

\begin{proof}
Suppose that $(X,f_{0,\infty})$ is multi-$\mathscr{F}$-sensitive with sensitivity constant $\delta>0$.
To prove that $(\mathcal{K}(X),\bar f_{0,\infty})$ is multi-$\mathscr{F}$-sensitive, it suffices to prove
$(\mathcal{K}_{F}(X),\bar f_{0,\infty})$ is multi-$\mathscr{F}$-sensitive by Lemma \ref{k1}.
Fix $k\in\mathbf{Z^{+}}$, $A_1=\{x_{1}^{1},\cdots,x_{m_1}^{1}\},\cdots,A_k=\{x_{1}^{k},\cdots,x_{m_k}^{k}\}\in \mathcal{K}(X)$ and $\epsilon>0$. By the multi-$\mathscr{F}$-sensitivity of $(X,f_{0,\infty})$,
\[\Gamma:=\bigcap_{j=1}^{k}\bigcap_{i=1}^{m_j}N_{d}(x_{i}^{j},\epsilon,\delta)\in\mathscr{F}.\]
Let $n\in\Gamma$. Then for any $1\leq j\leq k$ and $1\leq i\leq m_j$, there exists $y_{i}^{j}\in B_d(x_{i}^{j},\epsilon)$ such that
\[d(f_{0}^{n}(x_{i}^{j}),f_{0}^{n}(y_{i}^{j}))>\delta.\]
For any $1\leq j\leq k$, denote $B_j:=\{z_{i}^{j}\}_{i=1}^{m_j}$, where
\begin{equation*}
		z_{i}^{j}=\left\{\begin{array}{ll}
		y_{i}^{j},& {\rm if}\; d(f_{0}^{n}(x_{1}^{j}),f_{0}^{n}(x_{i}^{j}))\leq\delta/2,\\
        x_{i}^{j}, & {\rm otherwise}.
		\end{array}\right.
\end{equation*}
Then
\[D_X(f_{0}^{n}(A_j),f_{0}^{n}(B_j))\geq d(f_{0}^{n}(x_1^j),f_{0}^{n}(B_j))>\delta/2,\;1\leq j\leq k.\]
Note that $B_j\in B_{D_X}(A_j, \epsilon)$.
Thus, $n\in\bigcap_{j=1}^{k}N_{D_X}(A_{j},\epsilon,\delta/2)$
and $\Gamma\subset\bigcap_{j=1}^{k}N_{D_X}(A_{j},\epsilon,\delta/2)$.
which implies that $\bigcap_{j=1}^{k}N_{D_X}(A_{j},\epsilon,\delta/2)\in\mathscr{F}$.
Hence, $(\mathcal{K}_{F}(X),\bar f_{0,\infty})$ is multi-$\mathscr{F}$-sensitive,
and therefore, $(\mathcal{K}(X),\bar f_{0,\infty})$ is multi-$\mathscr{F}$-sensitive.

Suppose that $(\mathcal{K}(X),\bar f_{0,\infty})$ is multi-$\mathscr{F}$-sensitive with sensitivity constant $\delta>0$.
Let $k\in\mathbf{Z^{+}}$, $u_1,\cdots,u_k\in\mathcal{F}^1(X)$ and $\epsilon>0$.
By Lemma \ref{sen1} (ii),
\[\bigcap_{j=1}^{k}N_{D_X}([u_{j}]_{1},\epsilon/4,\delta)\subset\bigcap_{j=1}^{k}N_{d_{\infty}}(u_{j},\epsilon,\delta).\]
Thus, $\bigcap_{j=1}^{k}N_{d_{\infty}}(u_{j},\epsilon,\delta)\in\mathscr{F}$ since $\bigcap_{j=1}^{k}N_{D_X}([u_{j}]_{1},\epsilon/4,\delta)\in\mathscr{F}$.
Hence, $(\mathcal{F}^1(X),\tilde{f}_{0,\infty})$ is multi-$\mathscr{F}$-sensitive.

Suppose that $(\mathcal{F}^1(X),\tilde{f}_{0,\infty})$ is multi-$\mathscr{F}$-sensitive with sensitivity constant $\delta>0$.
Let $k\in\mathbf{Z^{+}}$, $x_1,\cdots,x_k\in X$ and $\epsilon>0$. By Lemma \ref{sen1} (i),
\[\bigcap_{j=1}^{k}N_{d_{\infty}}(\chi_{x_i},\epsilon,\delta)\subset\bigcap_{j=1}^{k}N_{D_X}(\{x_i\},\epsilon,\delta)
\subset\bigcap_{j=1}^{k}N_{d}(x_i,\epsilon,\delta).\]
Thus, $\bigcap_{j=1}^{k}N_{d}(x_i,\epsilon,\delta)\in\mathscr{F}$
since $\bigcap_{j=1}^{k}N_{d_{\infty}}(\chi_{x_i},\epsilon,\delta)\in\mathscr{F}$.
So, $(X,f_{0,\infty})$ is multi-$\mathscr{F}$-sensitive.
Therefore,  multi-$\mathscr{F}$-sensitivity among $(X,f_{0,\infty})$, $(\mathcal{K}(X),\bar f_{0,\infty})$ and $(\mathcal{F}^1(X),\tilde{f}_{0,\infty})$ are equivalent.
\end{proof}

\begin{remark}
{\rm(i)} The idea in the proof of Theorem \ref{sensitive} is partly motivated by Theorem 4 in \cite{Wu15}.

{\rm(ii)} $\mathscr{F}$-sensitivity among $(X,f_{0,\infty})$, $(\mathcal{K}(X),\bar f_{0,\infty})$
and $(\mathcal{F}^1(X),\tilde{f}_{0,\infty})$ are not totally equivalent.
Theorem 3.3 in \cite{Shao21} tells us that $(\mathcal{K}(X), \bar{f}_{0,\infty})$ is $\mathscr{F}$-sensitive if and only if
$(\mathcal{F}^{1}(X),$ $ \tilde f_{0,\infty})$ is $\mathscr{F}$-sensitive.
It is easy to verify that $\mathscr{F}$-sensitivity of $(\mathcal{K}(X), \bar{f}_{0,\infty})$
{\rm(}resp. $(\mathcal{F}^{1}(X),$ $ \tilde f_{0,\infty})${\rm)} implies that of $(X,f_{0,\infty})$
since $N_{D_X}(\{x\},\epsilon,\delta)\subset N_{d}(x,\epsilon,\delta)$ for any $x\in X$ and $\epsilon>0$,
where $\delta>0$ is a sensitivity constant of $(\mathcal{K}(X), \bar{f}_{0,\infty})$.
However, the converse is not true.
Let $\mathscr{F}_{inf}$ be the collection of all infinite subsets of $\mathbf{N}$.
$(X,f_{0,\infty})$ is sensitive if and only if $(X,f_{0,\infty})$ is $\mathscr{F}_{inf}$-sensitive
by Theorem 4.1 in \cite{Shao18}. Note that the sensitivity of $(X,f_{0,\infty})$ does not necessarily imply that of
$(\mathcal{K}(X), \bar{f}_{0,\infty})$, even for a single map $f$ {\rm(}see, for example, Example 2.11 in \cite{Sharman10}{\rm)}.
Therefore, The $\mathscr{F}_{inf}$-sensitivity of $(X,f_{0,\infty})$ does not necessarily imply that of
$(\mathcal{K}(X), \bar{f}_{0,\infty})$.
\end{remark}

Since multi-$\mathscr{F}$-sensitivity implies $\mathscr{F}$-sensitivity and they are equivalent when $\mathscr{F}$ is a filter,
we immediately get the following result.

\begin{Corollary}
Assume that $\mathscr{F}$ is a filter. Then the following statements are equivalent:

{\rm(i)} $(X,f_{0,\infty})$ is $\mathscr{F}$-sensitive;

{\rm(ii)} $(X,f_{0,\infty})$ is multi-$\mathscr{F}$-sensitive;

{\rm(iii)} $(\mathcal{K}(X),\bar f_{0,\infty})$ is $\mathscr{F}$-sensitive;

{\rm(iv)} $(\mathcal{K}(X),\bar f_{0,\infty})$ is multi-$\mathscr{F}$-sensitive;

{\rm(v)} $(\mathcal{F}^1(X),\tilde{f}_{0,\infty})$ is $\mathscr{F}$-sensitive;

{\rm(vi)} $(\mathcal{F}^1(X),\tilde{f}_{0,\infty})$ is multi-$\mathscr{F}$-sensitive.
\end{Corollary}

Since $\kappa\mathscr{F}_{inf}$ is a filter, we have the following results.

\begin{Corollary}
The following statements are equivalent:

{\rm(i)} $(X,f_{0,\infty})$ is $\kappa\mathscr{F}_{inf}$-sensitive;

{\rm(ii)} $(X,f_{0,\infty})$ is multi-$\kappa\mathscr{F}_{inf}$-sensitive;

{\rm(iii)} $(\mathcal{K}(X),\bar f_{0,\infty})$ is $\kappa\mathscr{F}_{inf}$-sensitive;

{\rm(iv)} $(\mathcal{K}(X),\bar f_{0,\infty})$ is multi-$\kappa\mathscr{F}_{inf}$-sensitive;

{\rm(v)} $(\mathcal{F}^1(X),\tilde{f}_{0,\infty})$ is $\kappa\mathscr{F}_{inf}$-sensitive;

{\rm(vi)} $(\mathcal{F}^1(X),\tilde{f}_{0,\infty})$ is multi-$\kappa\mathscr{F}_{inf}$-sensitive.
\end{Corollary}

By the fact that $(X,f_{0,\infty})$ is sensitive if and only if $(X,f_{0,\infty})$
is $\mathscr{F}_{inf}$-sensitive, together with Theorem 3.6 in \cite{Shao21},
we have the following result for an interval map.

\begin{Corollary}
The following statements are equivalent:

{\rm(i)} $(I,f)$ is sensitive;

{\rm(ii)} $(I,f)$ is $\mathscr{F}_{inf}$-sensitive;

{\rm(iii)} $(\mathcal{K}(I),\bar f)$ is sensitive;

{\rm(iii)} $(\mathcal{K}(I),\bar f)$ is $\mathscr{F}_{inf}$-sensitive;

{\rm(v)} $(\mathcal{F}^1(I),\tilde{f})$ is sensitive;

{\rm(vi)} $(\mathcal{F}^1(I),\tilde{f})$ is $\mathscr{F}_{inf}$-sensitive.
\end{Corollary}

\section*{Acknowledgement}

This research was partially supported by the Natural Science Foundation of Jiangsu Province of China (No. BK20200435)
and the Fundamental Research Funds for the Central Universities (No. NS2021053).

\end{CJK*}

\end{document}